\documentclass[10pt,english,reqno]{amsart} 

\usepackage[colorlinks=false]{hyperref}
\usepackage{graphicx}
\usepackage{amssymb}
\usepackage{epstopdf}
\usepackage{enumerate}
\usepackage{tikz-cd}
\usepackage{MnSymbol}

\DeclareMathSymbol{A}{\mathalpha}{operators}{`A}
\DeclareMathSymbol{B}{\mathalpha}{operators}{`B}
\DeclareMathSymbol{C}{\mathalpha}{operators}{`C}
\DeclareMathSymbol{D}{\mathalpha}{operators}{`D}
\DeclareMathSymbol{E}{\mathalpha}{operators}{`E}
\DeclareMathSymbol{F}{\mathalpha}{operators}{`F}
\DeclareMathSymbol{G}{\mathalpha}{operators}{`G}
\DeclareMathSymbol{H}{\mathalpha}{operators}{`H}
\DeclareMathSymbol{I}{\mathalpha}{operators}{`I}
\DeclareMathSymbol{J}{\mathalpha}{operators}{`J}
\DeclareMathSymbol{K}{\mathalpha}{operators}{`K}
\DeclareMathSymbol{L}{\mathalpha}{operators}{`L}
\DeclareMathSymbol{M}{\mathalpha}{operators}{`M}
\DeclareMathSymbol{N}{\mathalpha}{operators}{`N}
\DeclareMathSymbol{O}{\mathalpha}{operators}{`O}
\DeclareMathSymbol{P}{\mathalpha}{operators}{`P}
\DeclareMathSymbol{Q}{\mathalpha}{operators}{`Q}
\DeclareMathSymbol{R}{\mathalpha}{operators}{`R}
\DeclareMathSymbol{S}{\mathalpha}{operators}{`S}
\DeclareMathSymbol{T}{\mathalpha}{operators}{`T}
\DeclareMathSymbol{U}{\mathalpha}{operators}{`U}
\DeclareMathSymbol{V}{\mathalpha}{operators}{`V}
\DeclareMathSymbol{W}{\mathalpha}{operators}{`W}
\DeclareMathSymbol{X}{\mathalpha}{operators}{`X}
\DeclareMathSymbol{Y}{\mathalpha}{operators}{`Y}
\DeclareMathSymbol{Z}{\mathalpha}{operators}{`Z}

\usepackage[mathscr]{euscript}

\newcommand{\cal}{\mathscr}

\newcommand{\op}{\operatorname}
\newcommand{\tn}{\textnormal}

\newcommand{\Vect}{\mathrm{Vect}}

\newcommand{\QCoh}{\mathrm{QCoh}}

\newcommand{\coend}{\mathrm{coend}}

\DeclareMathOperator*\colim{colim}

\newcommand{\Hom}{\op{Hom}}
\newcommand{\Ker}{\op{Ker}}
\newcommand{\Aut}{\op{Aut}}
\newcommand{\End}{\op{End}}

\newcommand{\Sym}{\op{Sym}}

\newcommand{\Spec}{\op{Spec}}

\newcommand{\GL}{\mathrm{GL}}

\newcommand{\Rad}{\mathrm{Rad}}

\newtheorem{thm}[subsubsection]{Theorem}
\newtheorem*{thm*}{Theorem}
\newtheorem*{cor*}{Corollary}

\newtheorem{prop}[subsubsection]{Proposition}
\newtheorem{lem}[subsubsection]{Lemma}

\newtheorem{cor}[subsubsection]{Corollary}

\theoremstyle{definition}

\newtheorem{rem}[subsubsection]{Remark}

\numberwithin{equation}{section}

\newtheorem{untitledsubsubsection}[subsubsection]{}
\newtheorem{untitledsubsection}[subsection]{}

\newenvironment{sect}
{\begin{untitledsubsection}\bf}
{\end{untitledsubsection}}

\newenvironment{void}
{\begin{untitledsubsubsection}}
{\end{untitledsubsubsection}}

\title{Tannakian reconstruction of reductive group schemes}

\author{Yifei Zhao}
\thanks{This project has received funding from the European Research Council (ERC) under the European Union's Horizon 2020 research and innovation programme (grant agreement No.~851146).}
\date{\today}
\email{yifei.zhao@uni-muenster.de}

\begin{document}

\begin{abstract}
We give sharp criteria for when a reductive group scheme satisfies Tannakian reconstruction. When the base scheme is Noetherian, we explicitly identify its Tannaka group scheme.
\end{abstract}

\maketitle



\section{Introduction}

This note contains some observations on the category of finite-rank representations of a reductive group scheme.

To be precise, let $S$ be an affine scheme and $G \rightarrow S$ be a flat affine group scheme. Let $\Vect(S)^G$ denote the category of $G$-equivariant vector bundles on $S$, i.e.~finite projective $\cal O_S$-modules equipped with an $\cal O_G$-comodule structure. It embeds in the category $\QCoh(S)^G$ of $G$-equivariant quasi-coherent sheaves on $S$.

Write $\omega : \Vect(S)^G \rightarrow \Vect(S)$ for the forgetful functor. The presheaf $\underline{\Aut}^{\otimes}(\omega)$ of symmetric monoidal automorphisms of $\omega$ receives a natural map from $G$:
\begin{equation}
\label{eq-tannakian-map}
	G \rightarrow \underline{\Aut}^{\otimes}(\omega).
\end{equation}
It is known that \eqref{eq-tannakian-map} is an isomorphism when $S$ is a Dedekind domain, by classical Tannakian reconstruction of Saavedra, Deligne, and Milne \cite{MR0338002}, \cite{MR654325}, \cite{MR1106898}.

For a general affine scheme $S$, the morphism \eqref{eq-tannakian-map} may fail to be an isomorphism. The purpose of this note is to understand the source of this failure in the case of a reductive group scheme.

\begin{sect}
Summary of results
\end{sect}

\begin{void}
For any affine scheme $S$ and reductive group scheme $G\rightarrow S$, our Theorem \ref{thm-reconstruction-criterion} asserts that the following conditions are equivalent:
\begin{enumerate}
	\item $G$ satisfies the \emph{strong resolution property}, i.e.~every object of $\QCoh(S)^G$ is a $G$-equivariant quotient of a direct sum of objects in $\Vect(S)^G$;
	\item $G$ satisifes \emph{Tannakian reconstruction}, i.e.~\eqref{eq-tannakian-map} is an isomorphism;
	\item $G$ is \emph{linear}, i.e.~it is a closed subgroup scheme of $\GL_{n, S}\rightarrow S$ for some $n\ge 0$;
	\item The radical torus $\Rad(G)$ is \emph{isotrivial}, i.e.~it splits over a finite \'etale cover of $S$.
\end{enumerate}
\end{void}

\begin{void}
The implications (1) $\Rightarrow$ (2) $\Rightarrow$ (3) are established in much greater generality by Sch\"appi \cite[Corollary 7.5.2]{MR3185459}, although we supply a direct proof in the case of flat affine group schemes. The implication (3) $\Rightarrow$ (1) is due to Thomason \cite[Theorem 2.18]{MR893468} when $S$ is Noetherian and we explain the redundancy of this hypothesis. The equivalence (3) $\Leftrightarrow$ (4) is due to Gille \cite{gille2021reductive}, which we do not reproduce.

In \cite[\S8.2]{MR3185459}, Sch\"appi poses the following question: does there exist a flat affine group scheme such that $\Vect(S)^G$ does not generate $\QCoh(S)^G$ as an abelian category? This property is formally equivalent to the strong resolution property, so Theorem \ref{thm-reconstruction-criterion} answers Sch\"appi's question in the affirmative and produces explicit examples.
\end{void}

\begin{void}
When $S$ is furthermore connected and Noetherian, we determine the Tannaka group scheme $\underline{\Aut}^{\otimes}(\omega)$ of $G$.

To state the answer, we observe that the torus $\Rad(G)$ has a maximal isotrivial quotient $\Rad(G)\twoheadrightarrow \Rad(G)^f$. Let $G^f$ be the push-out of $G$ along this map. Then $G^f$ is representable by a reductive group scheme. Our Theorem \ref{thm-identification-tannaka-group} constructs a canonical isomorphism:
\begin{equation}
\label{eq-isotrivialization-reductive}
G^f \cong \underline{\Aut}^{\otimes}(\omega)
\end{equation}
of affine group schemes under $G$.

This result can be seen as a refinement of the equivalence between the isotriviality of $\Rad(G)$ and the Tannakian reconstruction of $G$. To my knowledge, it is the first instance where it is possible to explicitly identify a Tannaka group scheme which possibly differs from the original group scheme.
\end{void}

\begin{void}
This paper is organized as follows. Section \ref{sect-reconstruction} proves the equivalence among criteria for Tannakian reconstruction of a reductive group scheme (Theorem \ref{thm-reconstruction-criterion}). Section \ref{sect-tannaka} identifies the Tannaka group $\underline{\Aut}^{\otimes}(\omega)$ in the Noetherian setting (Theorem \ref{thm-identification-tannaka-group}).
\end{void}

\begin{sect}
Acknowledgements
\end{sect}

I thank Aise Johan de Jong for organizing the Stacks Project Workshop in 2020 and for leading the learning group on Tannakian formalism. I thank K\c{e}stutis \v{C}esnavi\v{c}ius and the anonymous referee for suggesting many references.

An earlier version of the paper contains a result characterizing Tannakian categories associated to flat group schemes satisfying the strong resolution property. This result is removed since it is subsumed by the works of Sch\"appi \cite{schappi2012characterization} \cite{MR4170218}. I thank the anonymous referee for pointing out my oversight.

\medskip

\section{Criteria for reconstruction}
\label{sect-reconstruction}

Let $S = \Spec(R)$ be an affine scheme and $G\rightarrow S$ be a flat affine group scheme. Hom-sets in the category $\QCoh(S)^G$ are denoted by $\Hom_G(-,-)$. We view $\cal O_G$ as an object of $\QCoh(S)^G$ via the group operation.

The goal of this section is to prove the following statement.

\begin{thm}
\label{thm-reconstruction-criterion}
If $G \rightarrow S$ is reductive, the following are equivalent:
\begin{enumerate}
	\item $G$ satisfies the strong resolution property;
	\item $G$ satisfies Tannakian reconstruction;
	\item $G$ is linear;
	\item $\Rad(G)$ is isotrivial.
\end{enumerate}
\end{thm}

The implications (1) $\Rightarrow$ (2) $\Rightarrow$ (3) are established in \S\ref{sect-general-implications}. The implication (3) $\Rightarrow$ (1) is the subject of \S\ref{sect-thomason-implication}. We quote \cite{gille2021reductive} for the equivalence (3) $\Leftrightarrow$ (4). Finally, we point out in Corollary \ref{cor-normal-domain} that these conditions are met when $S$ is a normal domain (not assumed Noetherian).

\begin{sect}
\label{sect-general-implications}
(1) $\Rightarrow$ (2) $\Rightarrow$ (3)
\end{sect}

\begin{void}
For any $\cal F\in\QCoh(S)^G$, consider the comma category $\Vect(S)^G_{/\cal F}$ of pairs $(\cal V, f)$ where $\cal V \in \Vect(S)^G$ and $f : \cal V\rightarrow\cal F$ is a morphism in $\QCoh(S)^G$. There is a canonical morphism:
\begin{equation}
\label{eq-canonical-colimit-expression}
L_{\cal F} : \colim_{(\cal V, f) \in \Vect(S)^G_{/\cal F}} \cal V \rightarrow \cal F.
\end{equation}

The implications (1) $\Rightarrow$ (2) $\Rightarrow$ (3) in Theorem \ref{thm-reconstruction-criterion} follow from the assertions below, which clarify the relationship among these conditions.
\end{void}

\begin{prop}
\label{prop-canonical-colimit-interpretation}
Let $G\rightarrow S$ be a flat affine group scheme. Then:
\begin{enumerate}[(a)]
	\item $G$ satisfies the strong resolution property if and only if $L_{\cal F}$ is bijective for all $\cal F\in\QCoh(S)^G$;
	\item $G$ satisfies Tannakian reconstruction if and only if $L_{\cal O_G}$ is bijective;
	\item when $G\rightarrow S$ is of finite type, $G$ is linear if and only if $L_{\cal O_G}$ is surjective.
\end{enumerate}
\end{prop}

\begin{proof}[Proof of Proposition \ref{prop-canonical-colimit-interpretation}(a)]
Since every colimit in $\QCoh(S)^G$ is a quotient of a direct sum, bijectivity of $L_{\cal F}$ for all $\cal F\in\QCoh(S)^G$ implies the strong resolution property.

To prove the converse, we first observe that $L_{\cal F}$ is surjective under the hypothesis. It remains to prove that it is injective. Since the index category $\Vect(S)^G_{/\cal F}$ contains finite direct sums, it suffices to show that for an individual object $(\cal V, f)\in\Vect(S)^G_{/\cal F}$, an element $v\in\cal V$ with $f(v) = 0$ vanishes in the colimit.

Since $G\rightarrow S$ is flat, the $R$-submodule $\Ker(f)\subset\cal V$ inherits a $G$-module structure. The strong resolution property gives some $\cal V_1 \in \Vect(S)^G$ with a morphism $\cal V_1\rightarrow \Ker(f)$ whose image contains $v$. The composition $\cal V_1\rightarrow\cal V\rightarrow\cal F$ vanishes, showing that the map:
$$
\cal V_1 \rightarrow\colim_{(\cal V, f) \in \Vect(S)^G_{/\cal F}}(\cal V)
$$
is zero, so in particular, $v$ vanishes in the colimit.
\end{proof}

\begin{void}
Before proving assertion (b), we record an observation: for each $\cal V\in\Vect(S)^G$, there is a canonical isomorphism between the $R$-module of $G$-equivariant maps $\cal V \rightarrow \cal O_G$ and the $R$-linear dual of $\cal V$:
\begin{equation}
\label{eq-equivariant-maps-dual}
	\Hom_G(\cal V, \cal O_G) \cong \cal V^{\vee}.
\end{equation}

Indeed, this map is defined by composing $f : \cal V \rightarrow \cal O_G$ with the counit $\epsilon : \cal O_G \rightarrow R$. Its inverse is given by composing the coaction map $\cal V \rightarrow \cal V\otimes\cal O_G$ with a given $\varphi \in\cal V^{\vee}$.
\end{void}

\begin{void}
Let $\omega : \Vect(S)^G \rightarrow \Vect(S)$ denote the forgetful functor. For any affine $S$-scheme $S'$, write $\omega_{S'}$ for the composition of $\omega$ with the natural functor $\Vect(S)\rightarrow\Vect(S')$.

The presheaf $\underline{\Aut}^{\otimes}(\omega)$ sending an affine $S$-scheme $S' = \Spec(R')$ to the group of automorphisms of $\omega_{S'}$ as a symmetric monoidal functor is representable by an affine group scheme (see \cite[\S4]{MR1106898} or \cite[\S2]{MR2101076}):
$$
\underline{\Aut}^{\otimes}(\omega) \cong \Spec(\coend(\omega^{\vee}\otimes_R \omega)).
$$
Here, $\omega^{\vee}\otimes_R \omega$ denotes the functor:
$$
(\Vect(S)^G)^{\mathrm{op}}\times \Vect(S)^G \rightarrow \QCoh(S),\quad \cal V_1, \cal V_2\mapsto (\cal V_1)^{\vee}\otimes_R\cal V_2,
$$
and $\coend(\omega^{\vee}\otimes_R \omega)$ is equipped with a natural Hopf algebra structure in $\QCoh(S)$.

There is a canonical map:
\begin{equation}
\label{eq-tannakian-unit}
G \rightarrow \underline{\Aut}^{\otimes}(\omega),
\end{equation}
sending an $S'$-point of $G$ to its action on $\cal V\otimes_R R'$ for all $\cal V\in\Vect(S)^G$. The condition that $G$ satisfies Tannakian reconstruction translates to the assertion that \eqref{eq-tannakian-unit} is an isomorphism of affine group schemes over $S$.
\end{void}

\begin{proof}[Proof of Proposition \ref{prop-canonical-colimit-interpretation}(b)]
To each object $(\cal V, f) \in \Vect(S)^G_{/\cal O_G}$, one may functorially attach a map of $R$-modules $\cal V \rightarrow \cal V^{\vee}\otimes_R\cal V$, $v\mapsto f^{\vee}\otimes v$ where $f^{\vee}\in\cal V^{\vee}$ corresponds to $f$ under \eqref{eq-equivariant-maps-dual}. Composing with the tautological map $\cal V^{\vee}\otimes_R\cal V \rightarrow \coend(\omega^{\vee}\otimes_R \omega)$, we obtain a morphism $\cal V \rightarrow \coend(\omega^{\vee}\otimes_R\omega)$. 

This process defines a map:
\begin{equation}
\label{eq-colimit-expression-of-coend}
	\colim_{(\cal V, f) \in \Vect(S)^G_{/\cal O_G}} \cal V \rightarrow \coend(\omega^{\vee}\otimes_R \omega).
\end{equation}
which we shall prove to be bijective.

Indeed, for any $\cal M\in\QCoh(S)$, a morphism from the coend to $\cal M$ is an $R$-linear natural transformation $\cal V\rightarrow \cal V\otimes_R\cal M$, $\cal V\in\Vect(S)^G$. A morphism from the colimit to $\cal M$ is a compatible system of $R$-linear maps $\cal V\rightarrow\cal M$ for each $f : \cal V\rightarrow\cal O_G$ in $\QCoh(S)^G$. The bijection between them is given by \eqref{eq-equivariant-maps-dual}.

To conclude, we observe that the morphism $L_{\cal O_G}$ corresponds to the canonical map $\coend(\omega^{\vee}\otimes_R \omega) \rightarrow \cal O_G$ under the isomorphism \eqref{eq-colimit-expression-of-coend}.
\end{proof}

\begin{void}
We introduce some notations to be used in the proof of Proposition \ref{prop-canonical-colimit-interpretation}(c).

Let $\cal M \in \Vect(S)$. The presheaf on $S$ which sends an affine $S$-scheme $S' = \Spec(R')$ to the group (resp.~monoid) of $R'$-linear automorphisms (resp.~endomorphisms) of $\cal M\otimes_R R'$ is representable by an affine group scheme $\GL(\cal M)\rightarrow S$ (resp.~$\underline{\End}(\cal M)\rightarrow S$).

Linearity of $G$ is equivalent to the condition of admitting a closed immersion of group schemes $G\hookrightarrow \GL(\cal M)$ for some $\cal M\in\Vect(S)$, because there always exists $\cal M'\in\Vect(S)$ such that $\cal M\oplus\cal M'$ is free.
\end{void}

\begin{void}
Given $\cal M\in\Vect(S)$, the following data are equivalent:
\begin{enumerate}
	\item a $G$-equivariance structure on $\cal M$;
	\item a morphism of monoid schemes $G \rightarrow \underline{\End}(\cal M)$ over $S$.
\end{enumerate}
Indeed, a $G$-equivariance structure on $\cal M$ is encoded by a coaction map $\cal M^{\vee}\otimes_R\cal M \rightarrow\cal O_G$, or a map of $R$-coalgebras $\Sym_R(\cal M^{\vee}\otimes_R\cal M) \rightarrow\cal O_G$.

Since $G$ is a group, any morphism of monoid schemes $G \rightarrow \underline{\End}(\cal M)$ factors through the open subscheme $\GL(\cal M)\subset \underline{\End}(\cal M)$.
\end{void}

\begin{proof}[Proof of Proposition \ref{prop-canonical-colimit-interpretation}(c)]
We borrow the isomorphism \eqref{eq-colimit-expression-of-coend} from the previous proof. It suffices to show that $G$ is linear if and only if the canonical map corresponding to \eqref{eq-tannakian-unit}:
\begin{equation}
\label{eq-assembling-coaction-maps}
\coend(\omega^{\vee}\otimes_R \omega) \rightarrow \cal O_G
\end{equation}
is surjective.

If $G$ is linear, then there exists some $\cal V \in \Vect(S)^G$ such that the coaction map $\cal V^{\vee}\otimes_R\cal V\rightarrow \cal O_G$ induces a surjection $\Sym_R(\cal V^{\vee}\otimes_R\cal V) \rightarrow \cal O_G$. This surjection factors through \eqref{eq-assembling-coaction-maps}, implying that surjectivity of the latter.

Conversely, note that $\coend(\omega^{\vee}\otimes_R\omega)$ is a quotient of $\bigoplus_{\cal V\in\Vect(S)^G}(\cal V^{\vee}\otimes_R\cal V)$.  Since $\Vect(S)^G$ admits finite direct sums and $\cal O_G$ is a finite type $R$-algebra, there exists some $\cal V\in\Vect(S)^G$ such that the image of $\cal V^{\vee}\otimes_R\cal V \rightarrow\cal O_G$ contains a set of generators of $\cal O_G$. This defines a closed immersion of monoid schemes $G\rightarrow \underline{\End}(\cal M)$, so $G$ is linear.
\end{proof}

\begin{rem}
\label{rem-strong-resolution-property}
For a flat affine group scheme $G \rightarrow S$, the strong resolution property has two additional equivalent characterizations:
\begin{enumerate}
	\item $\Vect(S)^G$ generates $\QCoh(S)^G$ as an abelian category: any morphism $f$ in $\QCoh(S)^G$ annihilated by $\Hom_G(\cal V, -)$ for all $\cal V\in\Vect(S)^G$ is necessarily zero.
	\item $\cal O_G \in \QCoh(S)^G$ is a filtered colimit of objects which belong to $\Vect(S)^G$. (Such $\cal O_G$ is known as an ``Adams Hopf algebra''.)
\end{enumerate}
See \cite[\S1.4]{MR2066503} and \cite[\S6.1]{schappi2012characterization} for a proof of these equivalences.
\end{rem}

\begin{sect}
\label{sect-thomason-implication}
(3) $\Rightarrow$ (1)
\end{sect}

\begin{void}
Suppose that $X$ is an $S$-scheme equipped with a $G$-action. Let $\QCoh(X)^G$ denote the category of $G$-equivariant quasi-coherent sheaves over $X$, and $\Vect(X)^G\subset\QCoh(X)^G$ the full subcategory of $G$-equivariant vector bundles.

We say that the pair $(G, X)$ satisfies the \emph{strong resolution property} if for every $\cal F \in \QCoh(X)^G$, there exists a family of objects $\cal V_{\alpha}\in\Vect(X)^G$ (for $\alpha\in A$) together with a $G$-equivariant surjection $\bigoplus_{\alpha\in A}\cal V_{\alpha}\twoheadrightarrow\cal F$.

In particular, the strong resolution property of $G$ is equivalent to that of the pair $(G, S)$.
\end{void}

\begin{void}
For an invertible sheaf $\cal L$ on $X$, we use the notion of being \emph{$S$-ample} as defined in \cite[01VG]{stacks-project}.

Let $f : X\rightarrow S$ denote the structure map. The existence of an $S$-ample invertible sheaf on $X$ implies that $f$ is quasi-compact and separated (\cite[01VI]{stacks-project}). In particular, the functor $f_* : \QCoh(X) \rightarrow \QCoh(S)$ is well-defined in this situation.
\end{void}

\begin{lem}
\label{lem-inheritance-resolution}
Suppose that $G$ satisfies the strong resolution property. Given any $S$-scheme $X$ equipped with a $G$-action which admits a $G$-equivariant, $S$-ample invertible sheaf, the pair $(G, X)$ satisfies the strong resolution property.
\end{lem}
\begin{proof}
Let $f : X\rightarrow S$ denote the structure map. Suppose $\cal F \in \QCoh(X)^G$. For each integer $k\ge 1$, the canonical morphism $f^*f_*(\cal F\otimes\cal L^{\otimes k}) \rightarrow \cal F\otimes \cal L^{\otimes k}$ is $G$-equivariant, where $f^*f_*(\cal F\otimes \cal L^{\otimes k})$ is equipped with the $G$-equivariance structure induced from that of $\cal F\otimes\cal L^{\otimes k}$.

Since $\cal L$ is $S$-ample, the induced map below is surjective (\cite[01Q3]{stacks-project}):
\begin{equation}
\label{eq-counit-morphism-sheaf}
\bigoplus_{k\ge 0} \cal L^{\otimes -k}\otimes f^*f_*(\cal F\otimes \cal L^{\otimes k}) \twoheadrightarrow \cal F.
\end{equation}
Because $G$ satisfies the strong resolution property, for each $k\ge 0$, there exists a family $\cal V_k^{(\alpha)} \in \tn{Vect}(S)^G$ (for $\alpha\in A_k$) with a surjection $\bigoplus_{\alpha\in A_k} \cal V_k^{(\alpha)} \twoheadrightarrow f_*(\cal F\otimes \cal L^{\otimes k})$. The composition:
$$
\bigoplus_{k\ge 0}\bigoplus_{\alpha\in A_k}\cal L^{\otimes -k}\otimes f^*\cal V_k^{(\alpha)} \twoheadrightarrow \bigoplus_{k\ge 0} \cal L^{\otimes -k}\otimes f^*f_*(\cal F\otimes\cal L^{\otimes k}) \twoheadrightarrow \cal F
$$
is the sought-for surjection from a sum of objects in $\tn{Vect}(X)^G$.
\end{proof}

\begin{lem}
\label{lem-inheritance-resolution-subgroup}
Suppose that $G$ is of finite presentation and satisfies the strong resolution property. Given a closed immersion $H\rightarrow G$ of flat affine group schemes such that $X:=G/H$ satisfies the hypothesis of Lemma \ref{lem-inheritance-resolution}, $H$ also satisfies the strong resolution property.
\end{lem}
\begin{proof}
The pair $(G, G/H)$ satisfies the strong resolution property by Lemma \ref{lem-inheritance-resolution}. Since $G\rightarrow G/H$ is faithfully flat and of finite presentation, the same holds for $G/H\rightarrow S$.

We have a commutative diagram of categories:
$$
\begin{tikzcd}
	\Vect(G/H)^G \ar[r, "\cong"]\ar[d] & \Vect(S)^H \ar[d]\\
	\QCoh(G/H)^G \ar[r, "\cong"] & \QCoh(S)^H
\end{tikzcd}
$$
where the horizontal functors are equivalences (fppf descent) and the vertical functors are fully faithful. The strong resolution property of $(G, G/H)$ thus implies that of $(H, S)$.
\end{proof}

\begin{void}
Recall that an affine group scheme $G\rightarrow S$ is \emph{reductive} if it is smooth with geometric fibers being connected reductive.

If $G\rightarrow S$ is reductive, then for any closed immersion of affine group schemes $G\hookrightarrow \GL_{n, S}$ over $S$, the quotient $\GL_{n, S}/G$ is representable by an affine $S$-schemes. This follows from \cite[Theorem 9.4.1 \& 9.7.5]{MR3272912}.
\end{void}

\begin{prop}[Thomason]
Suppose that $G \rightarrow S$ is reductive. If $G$ is linear, then it satisfies the strong resolution property.
\end{prop}
\begin{proof}
Lemma \ref{lem-inheritance-resolution-subgroup} reduces the problem to showing that $\GL_{n, S}$ satisfies the strong resolution property.

By Lemma \ref{lem-inheritance-resolution} applied to the morphism $S\rightarrow \Spec(\mathbb Z)$, it suffices to show that $\GL_{n, \Spec(\mathbb Z)}$ satisfies the strong resolution property. Since $\mathbb Z$ is a Dedekind domain, any flat affine group scheme over it satisfies the strong resolution property (\cite[Proposition 2 \& 3]{MR231831}).
\end{proof}

\begin{sect}
Additional remarks
\end{sect}

\begin{void}
\label{void-subgroups}
Suppose that $G$ is reductive and satisfies the equivalent conditions of Theorem \ref{thm-reconstruction-criterion}. Then any parabolic subgroup $P\subset G$ as well as the unipotent radical $N_P\subset P$ also satisfy the strong resolution property. Indeed, this follows from Lemma \ref{lem-inheritance-resolution-subgroup}.
\end{void}

\begin{cor}
\label{cor-normal-domain}
If $S$ is the spectrum of a normal domain, then any reductive group scheme $G \rightarrow S$, as well as its parabolic subgroups and their unipotent radicals, satisfy Tannakian reconstruction.
\end{cor}
\begin{proof}
Combine Theorem \ref{thm-reconstruction-criterion} with \cite[Lemma 2.2]{guo2020grothendieck}. For the statements on subgroups of $G$, we invoke the implication (1) $\Rightarrow$ (2) of Theorem \ref{thm-reconstruction-criterion}, which does not require the reductive hypothesis.
\end{proof}

\begin{rem}
Wedhorn \cite[\S5.17]{MR2101076} asserts that every flat affine group scheme over a valuation ring satisfies Tannakian reconstruction, but the proof contains a gap in \S5.6 of \emph{op.cit.}. This result gives a positive answer for reductive group schemes and their special subgroups.
\end{rem}

\medskip

\section{The Tannaka group scheme}
\label{sect-tannaka}

We assume that $S$ is an affine connected Noetherian scheme. This hypothesis guarantees that \'etale coverings of $S$ are locally Noetherian, so their connected components are open.

We study the maximal isotrivial quotient of tori in \S\ref{sect-isotrivialization}. Then we apply it to the radical torus of a reductive group schemes $G\rightarrow S$ to determine its Tannaka group scheme.

\begin{sect}
\label{sect-isotrivialization}
Maximal isotrivial quotients
\end{sect}

\begin{void}
Fix a geometric point $\bar s \rightarrow S$. Let $\Pi_1(S, \bar s)$ denote the ``pro-groupe fondamental \'elargi'' of \cite[X, \S10.6]{SGA3}. It pro-represents the functor sending an abstract group $\Gamma$ to the set of \'etale $\Gamma$-torsors rigidified along $\bar s$.

It follows from \cite[X, Th\'eor\`eme 7.1]{SGA3} that the functor $T \mapsto \Hom(T_{\bar s}, \mathbb G_{m, \bar s})$ defines an equivalence of categories between tori on $S$ and finite free $\mathbb Z$-modules equipped with a ``continuous'' $\Pi_1(S, \bar s)$-action, i.e.~one which factors through a group.

Under this equivalence, a torus $T$ is isotrivial if and only if the corresponding $\Pi_1(S, \bar s)$-action on $\Lambda := \Hom(T_{\bar s}, \mathbb G_{m,\bar s})$ factors through a finite group.
\end{void}

\begin{void}
Let $T \rightarrow S$ be a torus with associated $\Pi_1(S, \bar s)$-module $\Lambda$. Denote by $\Lambda^f\subset\Lambda$ the subset of elements whose $\Pi_1(S, \bar s)$-orbit is finite. Then $\Lambda^f\subset\Lambda$ is a $\mathbb Z$-submodule and $\Lambda/\Lambda^f$ is torsion-free. In particular, it induces a surjection of tori over $S$:
\begin{equation}
\label{eq-isotrivialization}
	T \twoheadrightarrow T^f.
\end{equation}

The torus $T^f$ is isotrivial and the morphism \eqref{eq-isotrivialization} is the universal morphism from $T$ to an isotrivial torus over $S$: it is the ``maximal isotrivial quotient'' of $T$.
\end{void}

\begin{rem}
Applying the same construction to $\check{\Lambda} := \Hom(\mathbb G_{m, \bar s}, T_{\bar s})$ also defines the ``maximal isotrivial subtorus'' of $T$.
\end{rem}

\begin{lem}
\label{lem-representation-isotrivialization}
Pulling back along \eqref{eq-isotrivialization} defines an equivalence of categories:
\begin{equation}
\label{eq-representation-isotrivialization}
	\Vect(S)^{T^f} \cong \Vect(S)^T.
\end{equation}
\end{lem}
\begin{proof}
Since \eqref{eq-isotrivialization} is surjective, the canonical functor $\Vect(S)^{T^f} \rightarrow \Vect(S)^T$ is fully faithful. It remains to prove essential surjectivity, i.e.~the $T$-action on any object $\cal V\in\Vect(S)^T$ factors through $T^f$.

Suppose that the $\Pi_1(S, \bar s)$-action on $\Lambda$ factors through a surjection $\Pi_1(S, \bar s) \twoheadrightarrow \Gamma$ where $\Gamma$ is a group (rather than a pro-group). We then obtain an \'etale $\Gamma$-torsor $S_1\rightarrow S$ rigidified along $\bar s$, i.e.~equipped with a lift $\bar s_1 \rightarrow S_1$ of $\bar s$.

The scheme $S_1$ is connected. Otherwise, we write $S_1'$ for the connected component containing $\bar s_1$. It is an \'etale $\Gamma'$-torsor for the subgroup $\Gamma'\subset\Gamma$ preserving $S_1'$. Furthermore, there is a canonical isomorphism $S_1'\times^{\Gamma'}\Gamma \cong S_1$, showing that $S_1$ is induced along $\Gamma'\subset\Gamma$ which contradicts the surjectivity of $\Pi_1(S, \bar s) \twoheadrightarrow\Gamma$.

By construction, the torus $T_1 := T\times_S S_1$ splits and there is a unique isomorphism:
\begin{equation}
\label{eq-splitting-torus-character}
\Hom(T_1, \mathbb G_{m, S_1}) \cong \Lambda,
\end{equation}
extending the isomorphism over $\bar s_1$. Thus, the base change $\cal V_1$ of $\cal V$ along $S_1\rightarrow S$ acquires a $\Lambda$-grading by $T_1$-weight submodules:
\begin{equation}
\label{eq-grading-inverse-image-of-representation}
	\cal V_1 \cong \bigoplus_{\lambda\in\Lambda} (\cal V_1)^{\lambda}.
\end{equation}
Since $\cal V_1$ is finite locally free, $(\cal V_1)^{\lambda} = 0$ for all but finitely many $\lambda$ and the rank of $(\cal V_1)^{\lambda}$ is constant along $S_1$ by connectedness.

The descent datum of $T_1$ gives rise to an isomorphism $T_{1, \bar s_1} \cong T_{1, \gamma(\bar s_1)}$ for all $\gamma\in\Gamma$. Under \eqref{eq-splitting-torus-character}, this isomorphism passes to the action map $\gamma : \Lambda\rightarrow\Lambda$. The descent datum of $\cal V_1$ as a $T_1$-representation gives rise to an isomorphism $\cal V_{1, \bar s_1} \cong \cal V_{1, \gamma(\bar s_1)}$ under which the weight-$\lambda$ submodule of $\cal V_{1, \bar s_1}$ corresponds to the weight-$\gamma(\lambda)$ submodule of $\cal V_{1, \gamma(\bar s_1)}$.

In summary, we find:
$$
(\cal V_1)_{\bar s_1}^{\lambda}\neq 0 \Leftrightarrow (\cal V_1)_{\gamma(\bar s_1)}^{\gamma(\lambda)} \neq 0 \Leftrightarrow (\cal V_1)_{\bar s_1}^{\gamma(\lambda)} \neq 0.
$$
Thus, if $(\cal V_1)^{\lambda} \neq 0$, the $\Gamma$-orbit of $\lambda$ is necessarily finite, i.e.~$\lambda\in\Lambda^f$.

The above argument shows that the $T_1$-action on $\cal V_1$ factors through $T_1^f$. This implies the same assertion about $\cal V$ since it is of \'etale local nature.
\end{proof}

\begin{void}
Let us illustrate this observation with Grothendieck's example of a non-isotrivial torus (\cite[X, \S1.6, Exemple 7.3]{SGA3}). We work over an algebraically closed field $k = \bar k$ and let $S := \mathbb A^1\sqcup_{\{0,1\}}\{0\}$ be the nodal cubic.

Since $\Pi_1(S, \bar s) \cong \mathbb Z$, its action on $\mathbb Z^{\oplus 2}$ by $a\cdot (x, y) = (x + a y, y)$ defines a rank-$2$ torus $T$ as a self-extension of $\mathbb G_m$:
\begin{equation}
\label{eq-nonisotrivial-torus}
1 \rightarrow \mathbb G_m \rightarrow T \rightarrow \mathbb G_m \rightarrow 1.
\end{equation}
The morphism \eqref{eq-isotrivialization} corresponds to the quotient morphism $T\twoheadrightarrow\mathbb G_m$ in \eqref{eq-nonisotrivial-torus}. Lemma \ref{lem-representation-isotrivialization} asserts that $T$-equivariant objects in $\Vect(S)$ are induced from $\mathbb G_m$-equivariant ones.
\end{void}

\begin{sect}
\label{sect-identification-tannaka}
Identification of $\underline{\Aut}^{\otimes}(\omega)$
\end{sect}

\begin{void}
Let $G \rightarrow S$ be a reductive group scheme. Specializing \eqref{eq-isotrivialization} to $\Rad(G)$, we obtain a surjection of tori $\Rad(G) \twoheadrightarrow \Rad(G)^f$.

Denote by $G^f$ the push-out of $G$ along this morphism:
\begin{equation}
\label{eq-isotrivialization-radical-torus-reductive}
G \twoheadrightarrow G^f.
\end{equation}
In other words, $G^f$ is the quotient of $G$ by the kernel $T_0$ of the map $\Rad(G) \twoheadrightarrow \Rad(G)^f$. Since $T_0$ is of multiplicative type and contained in the center of $G$, the quotient $G^f$ is representable by a reductive group scheme (\cite[Corollary 3.3.5]{MR3362641}) whose radical torus is identified with $\Rad(G)^f$.
\end{void}

\begin{lem}
\label{lem-representation-isotrivialization-reductive}
Pulling back along \eqref{eq-isotrivialization-radical-torus-reductive} defines an equivalence of categories:
\begin{equation}
\label{eq-representation-isotrivialization-reductive}
	\Vect(S)^{G^f} \cong \Vect(S)^G.
\end{equation}
\end{lem}
\begin{proof}
Since $G\twoheadrightarrow G^f$ is surjective, the functor $\Vect(S)^{G^f} \rightarrow \Vect(S)^G$ is fully faithful. It suffices to show essential surjectivity, i.e.~the $G$-action on any $\cal V\in\Vect(S)^G$ factors through $G^f$. This statement follows from Lemma \ref{lem-representation-isotrivialization}.
\end{proof}

\begin{thm}
\label{thm-identification-tannaka-group}
Let $S$ be an affine connected Noetherian scheme and $G \rightarrow S$ be a reductive group scheme. There is an isomorphism of affine group schemes under $G$:
\begin{equation}
\label{eq-identification-tannaka-group}
G^f \cong \underline{\Aut}^{\otimes}(\omega).
\end{equation}
\end{thm}
\begin{proof}
Let $\omega^f$ denote the symmetric monoidal functor $\Vect(S)^{G^f} \rightarrow \Vect(S)$. The naturality of \eqref{eq-tannakian-unit} yields a commutative diagram of affine group schemes:
$$
\begin{tikzcd}
	G \ar[d, twoheadrightarrow]\ar[r] & \underline{\Aut}^{\otimes}(\omega) \ar[d] \\
	G^f \ar[r] & \underline{\Aut}^{\otimes}(\omega^f)
\end{tikzcd}
$$
Lemma \ref{lem-representation-isotrivialization-reductive} shows that the right vertical arrow is an isomorphism. Theorem \ref{thm-reconstruction-criterion} shows that the bottom horizontal arrow is an isomorphism, since $\Rad(G)^f$ is isotrivial. The isomorphism \eqref{eq-identification-tannaka-group} thus follows.
\end{proof}

\bibliographystyle{amsalpha}
\bibliography{../biblio_mathscinet.bib}

\end{document}